\documentclass[12pt, reqno]{amsart}
\usepackage{amsmath, amsthm, amscd, amsfonts, amssymb, graphicx, color,url}
\usepackage[bookmarksnumbered, colorlinks, plainpages]{hyperref}  
\hypersetup{colorlinks=true,linkcolor=red, anchorcolor=green, citecolor=cyan, urlcolor=red, filecolor=magenta, pdftoolbar=true}

\newtheorem{theorem}{Theorem}[section]
\newtheorem{lemma}[theorem]{Lemma}
\newtheorem{proposition}[theorem]{Proposition}
\newtheorem{corollary}[theorem]{Corollary}
\theoremstyle{definition}
\newtheorem{definition}[theorem]{Definition}

\theoremstyle{remark}
\newtheorem{remark}[theorem]{Remark}
\numberwithin{equation}{section}

\begin{document}
\address{$^{[1,3]}$ Department of Mathematics, Jadavpur University, Kolkata 700032, West Bengal, India.}
\email{\url{pintubhunia5206@gmail.com} ; \url{kalloldada@gmail.com}}

\address{$^{[2_a]}$ University of Monastir, Faculty of Economic Sciences and Management of Mahdia, Mahdia, Tunisia}
\address{$^{[2_b]}$ Laboratory Physics-Mathematics and Applications (LR/13/ES-22), Faculty of Sciences of Sfax, University of Sfax, Sfax, Tunisia}
\email{\url{kais.feki@hotmail.com}\,;\,\url{kais.feki@fsegma.u-monastir.tn}}


\date{\today}
\author[P. Bhunia, K. Feki and K. Paul] {Pintu Bhunia$^{1}$, Kais Feki$^{2_{a,b}}$ and Kallol Paul$^{3}$ }
\title[Generalized $A$-numerical radius]{\Small{Generalized $A$-numerical radius of operators and related inequalities}}

\maketitle

\begin{abstract}
Let $A$ be a non-zero positive bounded linear operator on a complex Hilbert space $(\mathcal{H},\langle\cdot,\cdot\rangle)$. Let $\omega_A(T)$ denote the $A$-numerical radius of an operator $T$ acting on the semi-Hilbert space $(\mathcal{H},\langle\cdot,\cdot\rangle_A)$, where $\langle x, y\rangle_{A} :=\langle Ax, y\rangle$ for all $x,y\in \mathcal{H}$. Let $N_A(\cdot)$ be a seminorm on the algebra of all $A$-bounded operators acting on $\mathcal{H}$ and let $T$ be an operator which admits $A$-adjoint. Then, we define  the generalized $A$-numerical radius as 
$$\omega_{N_A}(T)=\displaystyle{\sup_{\theta \in \mathbb{R}}}\;
N_A\left(\frac{e^{i\theta}T+e^{-i\theta}T^{\sharp_A}}{2}\right),$$
where $T^{\sharp_A}$ denotes a distinguished $A$-adjoint of $T$. We develop several generalized $A$-numerical radius inequalities  from which follows  the existing  numerical radius and $A$-numerical radius inequalities. We also obtain bounds for generalized $A$-numerical radius of sum and product of operators. Finally, we study $\omega_{N_A}(\cdot)$ in the setting of two particular seminorms $N_A(\cdot)$.
\end{abstract}
\section{Introduction and Preliminaries}\label{S1}
 \noindent Let $\mathbb{B}(\mathcal{H})$  denote the $C^{\ast}$-algebra of all bounded linear operators acting on a complex Hilbert space $\mathcal{H}$ with with inner product $\langle \cdot,  \cdot\rangle$ and associated norm $\|\cdot\|$.  For $T\in\mathbb{B}(\mathcal{H})$, let $\|T\|=\sup\{\|Tx\|\,;\|x\| = 1\}$  and $
\omega(T)= \sup\{|\langle Tx,  x\rangle|\,; \|x\| = 1\}$
stand for the usual operator norm and the numerical radius of $T$, respectively.  Throughout this paper,  $A$ stands for  a non-zero positive bounded linear operator on $\mathcal{H}$, i.e., $\langle Ax,x \rangle \geq 0$ for all $x\in \mathcal{H}$. Clearly $A$ induces a  semi-inner product $\langle \cdot,  \cdot\rangle_A$ on $\mathcal{H}$ defined by ${\langle x, y\rangle}_A = \langle Ax, y\rangle$ for all $x,y\in \mathcal{H}$, and the semi-inner product  $\langle \cdot,  \cdot\rangle_A$ induces a seminorm $\|\cdot\|_A$ on $\mathcal{H}$ defined by $ \| x\|_A= \sqrt{\langle Ax,x \rangle}.$  It is easy to verify that $\|\cdot\|_A$ is a norm if and only if $A$ is injective. Also $(\mathcal{H}, \|\cdot\|_A)$ is complete space if and only if the range space $\mathcal{R}(A)$ of $A$ is closed in $\mathcal{H}$. By $\overline {\mathcal{R}(T)}$ we denote the norm closure of $\mathcal{R}(T)$ in $\mathcal{H}$.  Given $T \in \mathbb{B}(\mathcal{H})$   if  there exists a constant $c>0$ such that ${\|Tx\|}_{A} \le  c{\|x\|}_{A}$ for all $x\in \overline{\mathcal{R}(A)},$   then A-operator seminorm of $T$ is given by: 
$\|T\|_A=\sup_{x\in  \overline{\mathcal{R}(A)},\,x\neq 0}\frac{\|Tx\|_A}{\|x\|_A}< +\infty.$
The collection of all operators $T$ in $\mathbb{B}(\mathcal{H})$ with $\|T\|_A< +\infty$ is denoted by $\mathcal{B}^A(\mathcal{H}).$  In general, $\mathcal{B}^A(\mathcal{H})$  is not a subalgebra of $\mathbb{B}(\mathcal{H}).$  An operator $S \in \mathbb{B}(\mathcal{H})$ is called an $A$-adjoint of $T \in \mathbb{B}(\mathcal{H})$ if  $\langle Tx, y \rangle_A = \langle x, Sy \rangle_A$ for all $x,y\in \mathcal{H}$, that is, if $S$ is a solution of the operator equation $AX=T^*A$. The existence of an $A$-adjoint of $T$ is not guaranteed. Let $ \mathbb{B}_A(\mathcal{H})$ denote the collection of all operators in $\mathbb{B}(\mathcal{H})$, which admit $A$-adjoints.  By Douglas theorem \cite{doug}, it follows that 
$$ \mathbb{B}_A(\mathcal{H}) = \{ T \in \mathbb{B}(\mathcal{H}) : R(T^*A) \subseteq R(A) \}.$$ 
For $T\in \mathbb{B}_A(\mathcal{H}),$  the operator equation $AX=T^*A$ has a unique solution, denoted by $T^{\sharp_A}$, satisfying $\mathcal{R}(T^{\sharp_A})\subseteq \overline{\mathcal{R}(A)}$. The operator $T^{\sharp_A}$ has similar, but not identical, properties as the classical $T^*$. For instance, in general, $(T^{\sharp_A})^{\sharp_A}$ is not equal to $T$. More precisely, $(T^{\sharp_A})^{\sharp_A}=T$ if and only if $\mathcal{R}(T)\subseteq \overline{\mathcal{R}(A)}$.  Also $T^{\sharp_A}=A^\dag T^*A$, where $A^\dag $ is the Moore-Penrose inverse of $A$. 
Let  $\mathbb{B}_{A^{1/2}}(\mathcal{H})$  denote the set of all bounded linear operators that admit  $A^{1/2}$-adjoints. Then, by Douglas theorem \cite{doug},
$$\mathbb{B}_{A^{1/2}}(\mathcal{H}) = \Big\{T\in \mathbb{B}(\mathcal{H})\,; \,\, \exists\,\, c>0\,\text{ such that }
\,\,{\|Tx\|}_{A} \le  c{\|x\|}_{A}, \,\, \forall x\in \mathcal{H}\Big\}.$$
Operators in $ \mathbb{B}_{A^{1/2}}(\mathcal{H})$ are called $A$-bounded operators.  
Note that $\mathbb{B}_{A}(\mathcal{H})$ and $\mathbb{B}_{A^{1/2}}(\mathcal{H})$ are two subalgebras of $\mathbb{B}(\mathcal{H})$ which are neither closed nor dense in $\mathbb{B}(\mathcal{H})$. Moreover, the inclusion $\mathbb{B}_{A}(\mathcal{H}) \subseteq \mathbb{B}_{A^{1/2}}(\mathcal{H})$ holds, (see  \cite{acg1}). 
An operator $T\in\mathbb{B}_A(\mathcal{H})$ is said to be $A$-selfadjoint if $AT$ is selfadjoint, that is, $AT = T^*A.$
  $T$ is said to be $A$-positive if $AT$ is positive and we write $T\geq_{A}0$. Clearly $A$-positive operator is always  $A$-selfadjoint. For $A$-selfadjoint operators $T,S\in\mathbb{B}_A(\mathcal{H})$,  we write $T\geq_{A}S$ if and only if $T-S\geq_{A}0.$ An operator $T\in  \mathbb{B}_A(\mathcal{H})$ is said to be $A$-unitary if $\|Tx\|_A=\|T^{\sharp_A}x\|_A=\|x\|_A$ for all $x\in \mathcal{H}$, and $A$-normal if $T^{\sharp_A}T=TT^{\sharp_A}$ (see \cite{acg1,saddi}. Observe that for $T,S \in  \mathcal{B}_A(\mathcal{H})$,  $(TS)^{\sharp_A}=S^{\sharp_A}T^{\sharp_A}$, $\|Tx\|_A\leq \|T\|_A\|x\|_A$ for all $x\in \mathcal{H}$  and $\|TS\|_A\leq \|T\|_A\|S\|_A.$ 

So far we have discussed about $A$-operator seminorm induced by the positive operator $A$ and some of its important basic properties. It is time to talk of $A$-numerical radius in the setting of  sem-Hilbertian space operators. It is well known that the numerical radius plays an important role in various fields of operator theory and matrix analysis (see \cite{goldberg,halmos}). Several generalizations of numerical radius have  been studied (see \cite{kittaneh0,saddi}).  An important generalization of the numerical radius is the well known $A$-numerical radius of an operator $T\in \mathbb{B}(\mathcal{H})$,  introduced by Saddi \cite{saddi} as
\begin{align*}
	\omega_A(T) = \sup\big\{\big|{\langle Tx, x\rangle}_A\big|\;;\,\,x\in \mathcal{H}, \,{\|x\|}_A = 1\big\}.
\end{align*}
Several results on the $A$-numerical radius have been established by many mathematicians, (see  \cite{BFP,BNP1,feki01,fekijk,feki03,KS,ZMXF} and the references therein).   $\omega_A(\cdot)$ defines  a seminorm on $\mathbb{B}_{A^{1/2}}(\mathcal{H})$ which is equivalent to the $A$-operator seminorm $\|\cdot\|_A$. Moreover, for all $T\in\mathbb{B}_{A^{1/2}}(\mathcal{H})$, the following inequality holds
\begin{align}\label{refines}
\frac{1}{2}\|T\|_A \leq \omega_A(T)\leq \|T\|_A.
\end{align}
The above inequalities are sharp, $\frac{1}{2}\|T\|_A = \omega_A(T)$ if $AT^2=0$ and $\omega_A(T) = \|T\|_A$ if $T^{\sharp_A}T=TT^{\sharp_A}$.
Several refinements of \eqref{refines} have been recently established by many authors (see \cite{BFP,BPN,fekijk,feki03,zamani1} and the references therein).
Note that (see \cite{fg}) for $T\in\mathbb{B}_{A^{1/2}}(\mathcal{H})$,
\begin{equation}\label{newsemi}
\|T\|_A=\sup\left\{|\big\langle Tx, y\big\rangle_A|\,;\;x,y\in \mathcal{H},\,\|x\|_{A}=\|y\|_{A}= 1\right\}.
\end{equation}
For $T\in \mathbb{B}_A({\mathcal H})$, we write
$\Re_A(T):=\frac{T+T^{\sharp_A}}{2}\;\;\text{ and }\;\;\Im_A(T):=\frac{T-T^{\sharp_A}}{2\rm i}.$
Recently \cite[Th. 4.11]{bm} (see also \cite[Th. 2.5]{zamani1}) it is observed that for  $T\in \mathbb{B}_{A}(\mathcal{H}) $, 
\begin{align*}
\omega_A(T) = \displaystyle{\sup_{\theta \in \mathbb{R}}}{\left\|\Re_A(e^{i\theta}T)\right\|}_A,
\end{align*}
Another important generalization of the numerical radius  is   introduced recently by Abu-Omar and Kittaneh in \cite{kittaneh0} as follows. Given a norm $N(\cdot)$ on $ \mathbb{B}(\mathcal{H}) ,$  the generalized numerical radius of $T\in\mathbb{B}(\mathcal{H})$, denoted by $\omega_{N}(T)$, is defined by
$$\omega_{N}(T)=\displaystyle{\sup_{\theta \in \mathbb{R}}}\;
N\Big(\Re\big(e^{i\theta}T\big)\Big).$$
 It is easy to see that when $N(\cdot)=\|\cdot\|$, then $\omega_N(\cdot)$ coincides with the classical numerical radius $w(\cdot)$. The reader is invited to see \cite{kittaneh0,bc,ZMXF} for intermediate properties and  inequalities of the  norm $\omega_{N}(\cdot)$. The main objective of this paper is to study the generalized numerical radius in the setting of semi-Hilbertian space operators. Let $N_A(\cdot)$ be a seminorm on $\mathbb{B}_{A}(\mathcal{H}).$ For  $T\in \mathbb{B}_{A}(\mathcal{H}),$ we define
$$\omega_{N_A}(T)=\displaystyle{\sup_{\theta \in \mathbb{R}}}\;
N_A\Big(\Re_A\big(e^{i\theta}T\big)\Big).$$
Clearly, if $N_A(\cdot)$ is the $A$-operator seminorm, then $\omega_{N_A}(\cdot)$ coincides with the $A$-numerical radius $\omega_{A}(\cdot)$. One can also consider  $N_A(\cdot)$ to be a seminorm on $\mathbb{B}_{A^{1/2}}(\mathcal{H}).$                In this paper, we aim to study some basic properties of $\omega_{N_A}(\cdot)$ and develop several inequalities related to this new concept.  In Section \ref{sec2}, we study  generalized  $A$-numerical radius $\omega_{N_A}(\cdot)$ and develop several equalities and inequalities involving it. In section \ref{s3}, some inequalities related to $\omega_{N_A}(\cdot)$ involving the product and the sum of operators are obtained. In section \ref{s4}, we give two concrete examples of our new seminorm and develop inequalities related to them.

We end this section with the following known results which will be needed in due course of time.
\begin{proposition}\label{diez} \cite{acg1,acg2,fekijk}
	Let $T \in \mathbb{B}_A({\mathcal{H}})$. Then the following statements hold.
	\begin{itemize}
		\item [(i)] $T^{\sharp_A} \in \mathbb{B}_A({\mathcal{H}})$, $(T^{\sharp_A})^{\sharp_A}=P_{\overline{\mathcal{R}(A)}}TP_{\overline{\mathcal{R}(A)}}$, $((T^{\sharp_A})^{\sharp_A})^{\sharp_A}=T^{\sharp_A}$ and
		\begin{equation}\label{ddi}
			P_{\overline{\mathcal{R}(A)}}T^{\sharp_A}=T^{\sharp_A}P_{\overline{\mathcal{R}(A)}}=T^{\sharp_A}.
		\end{equation}
			\end{itemize}
		(Here $P_{\overline{\mathcal{R}(A)}}$ stands for the orthogonal projection onto ${\overline{\mathcal{R}(A)}}$.)
			\begin{itemize}
	\item [(ii)] $\|T^{\sharp_A}T\|_A = \| TT^{\sharp_A}\|_A=\|T\|_A^2 =\|T^{\sharp_A}\|_A^2$.
	\end{itemize}
\end{proposition}

\begin{proposition}\label{kk1}(\cite{bakfeki01,feki03})
	Let $T\in \mathbb{B}_A(\mathcal{H})$ be an $A$-selfadjoint operator. Then 
	\begin{itemize}
		\item [(i)] $T^{\sharp_A}$ is $A$-selfadjoint and $({T^{\sharp_A}})^{\sharp_A}=T^{\sharp_A}$.
		\item [(ii)] $\|T\|_{A}=\omega_A(T)$.
		\item [(iii)] $T^{2n}\geq_A 0$ for any positive integer $n$.
	\end{itemize}
\end{proposition}

\section{Generalized $A$-numerical radius}\label{sec2}

\noindent In this section, we begin with the definition of generalized $A$-numerical radius.
\begin{definition}
Let $N_A(\cdot)$ be a seminorm on $\mathbb{B}_{A}(\mathcal{H}).$  The function $\omega_{N_A}(\cdot)\,:\, \mathbb{B}_{A}(\mathcal{H})\to \mathbb{R}^{+}$, is defined as
\begin{equation}\label{def0}
\omega_{N_A}(T)=\displaystyle{\sup_{\theta \in \mathbb{R}}}\;
N_A\Big(\Re_A\big(e^{i\theta}T\big)\Big),
\end{equation}
for all $T\in \mathbb{B}_{A}(\mathcal{H})$.
\end{definition}
 \noindent It is easy to verify that $\omega_{N_A}(\cdot)$ defines a seminorm on $\mathbb{B}_{A}(\mathcal{H})$.
By replacing $T$ by ${\rm i} T$ in \eqref{def0}, one may see that
  \begin{equation}\label{def1}
\omega_{N_A}(T)=\displaystyle{\sup_{\theta \in \mathbb{R}}}\;
N_A\Big(\Im_A\big(e^{i\theta}T\big)\Big).
\end{equation}


\noindent We start our work with noting that $\max \{a,b\}=\frac{a+b}{2}+\frac{|a-b|}{2}$ for all $a,b\in \mathbb{R}.$ By using this elementary identity for real numbers, we first develop the following inequality.
\begin{theorem}\label{th1p}
	Let $T \in \mathbb{B}_A(\mathcal{H})$. Then
	\begin{eqnarray*}
		\omega_{N_A}(T)&\geq& \frac{ N_A(T)}{2} +  \frac{   \big|\,N_A(\Re_A(T))-N_A(\Im_A(T))                                                                                                                                                                                                                                                                                                                                                                                                                                                           \, \big|}{2}.
	\end{eqnarray*}
\end{theorem}
\begin{proof}
	From the definition of $\omega_{N_A}(\cdot)$, we have that $\omega_{N_A}(T) \geq N_A( \Re_A(T) )$ and $\omega_{N_A}(T) \geq  N_A( \Im_A(T) )$.
	Thus,
	\begin{eqnarray*}
		\omega_{N_A}(T)&\geq& \max \left \{  N_A( \Re_A(T) ),   N_A( \Im_A(T) )  \right \}\\
		&=& \frac{  N_A( \Re_A(T) )+  N_A( \Im_A(T) )}{2}+ \frac{\big|\, N_A( \Re_A(T) )- N_A( \Im_A(T) ) \, \big|}{2}\\
		&\geq& \frac{  N_A  \left ( \Re_A(T)  +{\rm i } \Im_A(T) \right )}{2}+ \frac{\big|\, N_A( \Re_A(T) )- N_A( \Im_A(T) ) \, \big|}{2}\\
		&=&  \frac{  N_A(T)}{2} +  \frac{   \big|\, N_A( \Re_A(T) )- N_A( \Im_A(T) )\, \big|}{2}.
	\end{eqnarray*}	
This completes the proof.	
\end{proof}

\begin{remark}\label{r1}
	For all $\theta \in \mathbb{R}$, we have, $\omega_{N_A}(e^{\rm i \theta}T)=\omega_{N_A}(T)$  and $N_A(e^{\rm i \theta}T)=N_A(T)$. Therefore, it follows from Theorem \ref{th1p} that
	\begin{eqnarray}\label{1eqn}
	\omega_{N_A}(T)&\geq& \frac{ N_A(T)}{2} + \sup_{\theta\in \mathbb{R}} \frac{   \big|\,N_A(\Re_A(e^{\rm i \theta}T))-N_A(\Im_A(e^{\rm i \theta}T))                                                                                                                                                                                                                                                                                                                                                                                                                                                           \, \big|}{2}.
	\end{eqnarray}
\end{remark}

\noindent In our next theorem, by using Remark \ref{r1} we prove the following equivalent conditions.

\begin{theorem}\label{th2p}
	Let $T \in \mathbb{B}_A(\mathcal{H})$.  Then the following conditions are equivalent. \\
	(i) $\omega_{N_A}(T)=\frac{N_A(T)}{2}$.\\
	(ii) $N_A\left ( \Re_A(e^{\rm i \theta }T) \right )=N_A\left ( \Im_A(e^{\rm i \theta }T )\right )=\frac{N_A(T)}{2}$ for all $\theta \in \mathbb{R}.$
\end{theorem}

\begin{proof}
	We only prove (i) $\implies$ (ii), as	(ii)  $\implies$ (i) is trivial.
	Let $\omega_{N_A}(T)=\frac{ N_A(T)}{2}$. Then it follows from (\ref{1eqn}) that $ N_A( \Re_A(e^{\rm i \theta }T) )= N_A( \Im_A(e^{\rm i \theta }T) )\,\, \textit{for all $\theta \in \mathbb{R}$}.$ Therefore, for all $\theta \in \mathbb{R}$,  we have
	\begin{eqnarray*}
		N_A( \Re_A(e^{\rm i \theta }T) )&\leq& \omega_{N_A}(T)=\frac{  N_A(T) }{2}=\frac{  N_A(e^{\rm i \theta }T) }{2}\\
		&=&\frac{ N_A\left ( \Re_A(e^{\rm i \theta }T)+\rm i \Im_A(e^{\rm i \theta }T)\right )}{2}\\
		&\leq & \frac{ N_A( \Re_A(e^{\rm i \theta }T) )+ N_A( \Im_A(e^{\rm i \theta }T) )}{2} \\
		&   =&  N_A( \Re_A(e^{\rm i \theta }T) ).
	\end{eqnarray*}
	This implies that $N_A\left ( \Re_A(e^{\rm i \theta }T )\right )=N_A\left ( \Im_A(e^{\rm i \theta }T) \right )=\frac{N_A(T)}{2}$  for all $\theta \in \mathbb{R}.$	
\end{proof}

\noindent The following definition is useful for the sequel.
\begin{definition}
 The seminorm  $N_A(\cdot)$ is said to 
\begin{itemize}
  \item [(1)] be submultiplicative if $N_A(TS)\leq N_A(T)N_A(S)$ for all $T,S\in\mathbb{B}_{A^{1/2}}(\mathcal{H})$.
  \item [(2)] be $A$-selfadjoint invariant if $N_A(T)= N_A(T^{\sharp_A})$ for all $T\in\mathbb{B}_{A}(\mathcal{H})$.
  \item [(3)] be $A$-weakly unitarily invariant if $N_A(U^{\sharp_A}TU)=N_A(T)$ for all $T\in\mathbb{B}_{A^{1/2}}(\mathcal{H})$ and $A$-unitary operator $U\in\mathbb{B}_A(\mathcal{H})$.
  \item [(4)]  be $A$-increasing if for every $A$-selfadjoint operators $T,S\in\mathbb{B}_{A}(\mathcal{H})$ ,
  $$T\geq_AS \Longrightarrow N_A(T)\geq N_A(S).$$
  \item [(5)] satisfy   $A$-power property if $N_A(T^n)=N_A^n(T)$ for every $A$-selfadjoint operator $T$ and every positive integer $n$.
\end{itemize}
\end{definition}

\noindent In the following proposition we prove two basic properties related to $\omega_{N_A}(\cdot)$.
\begin{proposition}\label{0p1}
	Let $T\in \mathbb{B}_A(\mathcal{H})$. If $N_A(\cdot)$ is $A$-selfadjoint invariant, then 
\begin{equation}\label{prop1p}
\omega_{N_A}(T)=\omega_{N_A}(T^{\sharp_A}).
\end{equation}
If, in addition, $N_A(\cdot)$ is $A$-weakly unitarily invariant, then so is $\omega_{N_A}(\cdot)$.
\end{proposition}
\begin{proof}
One observes that,
	\begin{eqnarray*}
		\omega_{N_A}(T)&=& \sup_{\theta \in \mathbb{R}} N_A\left ( \Re_A \left(e^{\rm i \theta}T\right)\right)\\
		&=& \sup_{\theta \in \mathbb{R}} N_A  \left(\left ( \Re_A \left(e^{\rm i \theta}T\right)\right)^{\sharp_A}\right) \,\, \Big ( \text{since $N_A(\cdot)$ is $A$-selfadjoint invariant} \Big )\\
		&=& \sup_{\theta \in \mathbb{R}} N_A\left ( \Re_A \left(e^{-\rm i \theta}T^{\sharp_A}\right)\right)\\
		&=& \omega_{N_A}(T^{\sharp_A}).
	\end{eqnarray*}
Now, let $U$ be an $A$-unitary operator. By using \eqref{prop1p}, we see that
	\begin{align*}
		\omega_{N_A}(U^{\sharp_A}TU)
&=\omega_{N_A}\big(U^{\sharp_A}T^{\sharp_A}(U^{\sharp_A})^{\sharp_A}\big)\\
&= \sup_{\theta \in \mathbb{R}}\, N_A\left ( \Re_A \left(e^{\rm i \theta}\big(U^{\sharp_A}T^{\sharp_A}(U^{\sharp_A})^{\sharp_A}\big)\right)\right)\\
		&= \sup_{\theta \in \mathbb{R}}\,N_A\Big(U^{\sharp_A}\left( \Re_A \left(e^{\rm i \theta}T^{\sharp_A}\right)\right)(U^{\sharp_A})^{\sharp_A}\Big)\\
		&= \sup_{\theta \in \mathbb{R}}\,N_A\Big(U^{\sharp_A}\left( \Re_A \left(e^{\rm i \theta}T^{\sharp_A}\right)\right)^{\sharp_A} (U^{\sharp_A})^{\sharp_A}\Big)\\
		&=\sup_{\theta \in \mathbb{R}}\,N_A     \left (\Big(U^{\sharp_A}\left( \Re_A \left(e^{\rm i \theta}T^{\sharp_A}\right)\right)U \Big)^{\sharp_A} \right )\\
			&=\sup_{\theta \in \mathbb{R}}\,N_A     \Big(U^{\sharp_A}\left( \Re_A \left(e^{\rm i \theta}T^{\sharp_A}\right)\right)U \Big)
\end{align*}
where the last equality follows since $N_A(\cdot)$ is $A$-selfadjoint invariant. Furthermore, since $N_A(\cdot)$ is $A$-weakly unitarily invariant. So, we get
	\begin{align*}
		\omega_{N_A}(U^{\sharp_A}TU)
&=\sup_{\theta \in \mathbb{R}}\,N_A\Big( \Re_A \left(e^{\rm i \theta}T^{\sharp_A}\right)\Big).
	\end{align*}
Thus, in the virtue of \eqref{prop1p}, we obtain
$$\omega_{N_A}(U^{\sharp_A}TU)=\omega_{N_A}(T^{\sharp_A})=\omega_{N_A}(T).$$
Thus, the proof is complete.
\end{proof}

\noindent Our next result is stated as follows.
\begin{theorem}\label{th5p}
	Let $T\in \mathbb{B}_A(\mathcal{H})$. Then, $$\frac{ 1}{2}N_A(T)\leq \omega_{N_A}(T).$$
	In particular, if  $N_A(\cdot)$ is $A$-selfadjoint invariant, then
	\begin{equation}\label{666}
	\frac{1}{2}N_A(T)\leq \omega_{N_A}(T)\leq N_A(T).
	\end{equation}
\end{theorem}
\begin{proof} We only prove that $\omega_{N_A}(T)\leq N_A(T)$ if $N_A(\cdot)$ is $A$-selfadjoint invariant, since other inequalities follows trivially from Remark \ref{r1}. Let $N_A(\cdot)$ be $A$-selfadjoint invariant and let
	$\theta \in \mathbb{R}$. Then,
	\begin{eqnarray*}
		N_A\left (\Re_A(e^{\rm i \theta}T)   \right)=N_A\left(\frac{e^{i\theta}T+e^{-i\theta}T^{\sharp_A}}{2}\right)
		&\leq &\frac{1}{2} \left (N_A(e^{\rm i \theta}T)+ N_A(e^{-\rm i \theta}T^{\sharp_A})  \right )\\
		&=&\frac{1}{2} \left (N_A(T)+ N_A(T^{\sharp_A})  \right )\\
		&=& N_A(T).
	\end{eqnarray*}
	Therefore, by considering supremum over all $\theta \in \mathbb{R}$, we get the required inequality.
\end{proof}

\noindent Again, by using the maximum function of two real numbers, we  obtain another lower bound for $\omega_{N_A}(\cdot),$ assuming  $N_A(\cdot)$ to be submultiplicative.

\begin{theorem}\label{th3p}
	Let $T \in \mathbb{B}_A(\mathcal{H})$. If $N_A(\cdot)$ is submultiplicative, then
	\begin{eqnarray*}
		\omega_{N_A}(T)&\geq& \sqrt{\frac{1}{4} N_A\left ( T^{\sharp_A}T+TT^{\sharp_A}\right ) +  \frac{1}{2}\left|\, N_A^2( \Re_A(T) )- N_A^2( \Im_A(T) )\, \right|  }.
	\end{eqnarray*}
\end{theorem}

\begin{proof}
	Following the proof of Theorem \ref{th1p}, we have that
	\begin{eqnarray*}
		w^2_{N_A}(T)&\geq& \max \left\{ N_A^2( \Re_A(T) ),   N_A^2( \Im_A(T) ) \right\}\\
		&=& \frac{ N_A^2( \Re_A(T) )+   N_A^2( \Im_A(T) )}{2}+ \frac{  \left|  N_A^2( \Re_A(T) )-  N_A^2( \Im_A(T) )\right|}{2} \\
		&\geq & \frac{ N_A \left ( [\Re_A(T)]^2 \right )+ N_A \left ( [\Im_A(T)]^2 \right ) }{2}+ \frac{1}{2}\left|\, N_A^2( \Re_A(T) )- N_A^2( \Im_A(T) )\, \right|\\
		&\geq& \frac{ N_A \left (  [\Re_A(T)]^2 +[\Im_A(T)]^2  \right)}{2}+ \frac{1}{2}\left|\, N_A^2( \Re_A(T) )- N_A^2( \Im_A(T) )\, \right|\\
		&=& \frac{1}{4}  N_A\left ( T^{\sharp_A}T+TT^{\sharp_A}\right )+ \frac{1}{2}\left|\, N_A^2( \Re_A(T) )- N_A^2( \Im_A(T) )\, \right|.
	\end{eqnarray*}
This proves the desired inequality.	
\end{proof}

\begin{remark}\label{r2}
	For all $\theta \in \mathbb{R}$, we have, $\omega_{N_A}(e^{\rm i \theta}T)=\omega_{N_A}(T)$. Therefore, it follows from Theorem \ref{th3p} that
	\begin{eqnarray*}
		\omega_{N_A}(T)&\geq& \sqrt{\frac{1}{4} N_A\left ( T^{\sharp_A}T+TT^{\sharp_A}\right ) +  \frac{1}{2} \sup_{\theta \in \mathbb{R}} \left|\, N_A^2( \Re_A(e^{\rm i \theta}T) )- N_A^2( \Im_A(e^{\rm i \theta}T) )\, \right|  }.
	\end{eqnarray*}
\end{remark}

\noindent The next result is a generalization of the well known inequality obtained by Kittaneh \cite[Th. 1]{FK}.
\begin{theorem}
	Let $T\in \mathbb{B}_{A}(\mathcal{H})$ and let $N_A(\cdot)$ be submultiplicative. Then,
	\begin{equation*}\label{140}
	\frac{1}{2}\sqrt{N_A\left(T^{\sharp_A} T+TT^{\sharp_A}\right)}\le  \omega_{N_A}\left(T\right).
	\end{equation*}
	Moreover, if $N_A(\cdot)$ is $A$-increasing and satisfies $A$-power property, then
	\begin{equation*}\label{141}
	\frac{1}{2}\sqrt{N_A\left(T^{\sharp_A} T+TT^{\sharp_A}\right)}\le  \omega_{N_A}\left(T\right)\leq \frac{\sqrt{2}}{2}\sqrt{N_A\left(T^{\sharp_A} T+TT^{\sharp_A}\right)}.
	\end{equation*}
\end{theorem}
\begin{proof}
	We only prove that $\omega_{N_A}\left(T\right)\leq \frac{\sqrt{2}}{2}\sqrt{N_A\left(T^{\sharp_A} T+TT^{\sharp_A}\right)}$ when $N_A(\cdot)$ is $A$-increasing and satisfies $A$-power property, other inequality follows from Remark \ref{r2}.
	Let $\theta \in \mathbb{R}$.
Clearly, $\Im_A(e^{i\theta}T)$ is an $A$-selfadjoint operator. Therefore, it follows from Proposition \ref{kk1} (iii) that $[\Im_A(e^{i\theta}T)]^2\geq_A 0$. So, we infer that
	\begin{align*}
	\frac{1}{2}\left(TT^{\sharp_A} + T^{\sharp_A} T\right)-\left[\Re_A(e^{i\theta}T)\right]^2=\left[\Im_A(e^{i\theta}T)\right]^2\geq_A 0.
	\end{align*}
	This implies that
	\begin{align*}
	\frac{1}{2}\left(TT^{\sharp_A} + T^{\sharp_A} T\right)\geq_A\left[\Re_A(e^{i\theta}T)\right]^2.
	\end{align*}
	Since $N_A(\cdot)$ is $A$-increasing, so we get
	\begin{align*}
	N_A\Big(\left[\Re_A(e^{i\theta}T)\right]^2\Big)\leq \frac{1}{2}N_A\left(TT^{\sharp_A} + T^{\sharp_A} T\right).
	\end{align*}
Also, since $N_A(\cdot)$ satisfies $A$-power property, so we have
	\begin{align*}
	N_A^2\Big(\Re_A(e^{i\theta}T)\Big)\leq \frac{1}{2}N_A\left(TT^{\sharp_A} + T^{\sharp_A} T\right).
	\end{align*}
Now by taking supremum over all $\theta\in \mathbb{R}$, we obtain that
	\begin{align*}\label{n2}
	\omega_{N_A}^2(T)\leq\frac{1}{2}N_A\left(TT^{\sharp_A} + T^{\sharp_A} T\right).
	\end{align*}
This completes the proof.
\end{proof}

\noindent Next, we give a complete characterization for $\frac{1}{2}\sqrt{N_A\left ( T^{\sharp_A}T+TT^{\sharp_A}\right ) }=\omega_{N_A}(T).$

\begin{theorem}\label{th4p}
	Let $T\in \mathbb{B}_A(\mathcal{H})$. If  $N_A(\cdot)$ is submultiplicative, then the following conditions are equivalent.\\
	(i) $\omega_{N_A}(T)=\frac{1}{2}\sqrt{N_A\left ( T^{\sharp_A}T+TT^{\sharp_A}\right ) }$.\\
	(ii) $N_A\left ( \Re_A(e^{\rm i \theta }T \right )=N_A\left ( \Im_A(e^{\rm i \theta }T \right )=\frac{1}{2}\sqrt{N_A\left ( T^{\sharp_A}T+TT^{\sharp_A}\right ) }$ for all $\theta \in \mathbb{R}.$
	
\end{theorem}

\begin{proof}
	We only prove (i) $\implies $ (ii), as	(ii)  $\implies$ (i) is trivial.  Let $\omega_{N_A}(T)=\sqrt{\frac{1}{4}N_A\left ( T^{\sharp_A}T+TT^{\sharp_A}\right ) }$. Then, for all $\theta\in \mathbb{R}$, we have
	\begin{eqnarray*}
		\frac{1}{4}N_A\left ( T^{\sharp_A}T+TT^{\sharp_A}\right ) &=& \frac{1}{2} N_A\left ( \left(\Re_A(e^{\rm i \theta}T)\right )^2+\left(\Im_A(e^{\rm i \theta}T)\right)^2 \right ) \\
		&\leq& \frac{1}{2} \left( N_A^2 \left ( \Re_A(e^{\rm i \theta}T) \right )+ N_A^2 \left ( \Im_A(e^{\rm i \theta}T)\right )\right)\\
		&\leq& \omega_{N_A}^2(T)\\
		&=& \frac{1}{4}N_A\left ( T^{\sharp_A}T+TT^{\sharp_A}\right ).
	\end{eqnarray*}
	Thus, $N_A^2\left ( \Re_A(e^{\rm i \theta }T )\right )=N_A^2\left ( \Im_A(e^{\rm i \theta }T )\right )=\omega_A^2(T)={\frac{1}{4}N_A\left ( T^{\sharp_A}T+TT^{\sharp_A}\right ) }$  for all $\theta \in \mathbb{R}$.
\end{proof}

\noindent Now, in the following theorem we obtain an upper bound for $\omega_{N_A}(T)$.

\begin{theorem}\label{th6p}
	Let $T\in \mathbb{B}_A(\mathcal{H})$. Then
	\[\omega_{N_A}(T) \leq \sqrt{  N_A^2\left(\Re_A(T )\right)+ N_A^2\left(\Im_A(T )\right)}.\]
\end{theorem}

\begin{proof}
	Let $\theta \in \mathbb{R}$. Then,
	$N_A\left (  \Re_A({e^{\rm i \theta}}T) \right )= N_A \left ( \alpha \Re_A(T)-\beta \Im_A(T) \right ),$
	where $\alpha=\cos\theta$ and $\beta=\sin \theta.$ Now,
	\begin{eqnarray*}
		\omega_{N_A}(T)&=&\sup_{\theta\in \mathbb{R}} N_A\left (  \Re_A(e^{e^{\rm i \theta}}T) \right )\\
		&=&\sup \left \{N_A(   \alpha \Re_A(T)+\beta \Im_A(T) ): \alpha, \beta \in \mathbb{R}, \alpha^2+\beta^2=1  \right \} .
	\end{eqnarray*}
	Therefore, for $ \alpha, \beta \in \mathbb{R}$,	we have
	\begin{eqnarray*}
		&& \sup_{\alpha^2+\beta^2=1}N_A(   \alpha \Re_A(T)+\beta \Im_A(T) )\\
		&\leq & \sup_{\alpha^2+\beta^2=1} \left ( | \alpha |N_A(  \Re_A(T) )+|\beta| N_A( \Im_A(T) ) \right )\\
		&\leq& \sup_{\alpha^2+\beta^2=1} \sqrt{ (\alpha^2+\beta^2) (N_A^2(  \Re_A(T) )+ N_A^2( \Im_A(T) ) )   }  \\
		&=&\sqrt{  N_A^2(  \Re_A(T) )+ N_A^2( \Im_A(T) )    }.
	\end{eqnarray*}
	This completes the proof.
\end{proof}

\begin{remark}
	For all $\theta \in \mathbb{R}$, we have $\omega_{N_A}(e^{\rm i \theta}T)=\omega_{N_A}(T)$. Therefore, it follows from Theorem \ref{th6p} that
	\begin{eqnarray*}
	\omega_{N_A}(T) \leq \inf_{\theta \in \mathbb{R}}\sqrt{  N_A^2\left(\Re_A(e^{\rm i \theta}T )\right)+ N_A^2\left(\Im_A(e^{\rm i \theta}T )\right)}.
	\end{eqnarray*}
\end{remark}

\noindent In the next two theorems, we establish upper bounds related to $\omega_{N_A}(\cdot)$ under suitable conditions.

\begin{theorem}\label{th9p}
	Let $T\in \mathbb{B}_A(\mathcal{H})$. If $N_A(\cdot)$ satisfies $A$-power property, then
	\[\omega_{N_A}(T)\leq \sqrt{  \frac{1}{2}\omega_{N_A}(T^2)+\frac{1}{4}N_A(T^{\sharp_A}T+TT^{\sharp_A})}.\]
\end{theorem}

\begin{proof}
For  $\theta \in \mathbb{R}$  we have,
	\begin{eqnarray*}
		N_A \left (\left[\Re_A\left(e^{\rm i \theta}T\right)\right]^2 \right )&=& N_A \left (\left(\frac{e^{\rm i \theta}T+e^{-\rm i \theta}T^{\sharp_A}}{2}\right)^2 \right )\\
		&=& N_A\left(\frac{2\Re_A\left(e^{2\rm i \theta}T^2\right)+T^{\sharp_A}T+TT^{\sharp_A}}{4}\right)\\
		&\leq& \frac{1}{2} N_A\left(\Re_A\left (e^{2\rm i \theta}T^2 \right)\right) +\frac{1}{4}N_A\left(T^{\sharp_A}T+TT^{\sharp_A}\right)\\
		&\leq& \frac{1}{2} \omega_{N_A} \left(T^2\right)  +\frac{1}{4}N_A\left(T^{\sharp_A}T+TT^{\sharp_A}\right).
	\end{eqnarray*}
	Since, $\Re_A\left(e^{\rm i \theta}T\right)$ is $A$-selfadjoint and $N_A(\cdot)$ satisfies $A$-power property, so
	\begin{eqnarray*}
		N_A^2 \left(\Re_A\left(e^{\rm i \theta}T\right)\right) 	&\leq& \frac{1}{2} \omega_{N_A} \left(T^2\right)  +\frac{1}{4}N_A\left(T^{\sharp_A}T+TT^{\sharp_A}\right).
	\end{eqnarray*}
	Taking supremum over all $\theta\in \mathbb{R}$, we get the required inequality.
\end{proof}

\begin{theorem} \label {theorem:upper7}
	Let $T \in \mathbb{B}_A(\mathcal{H}).$ If $N_A(\cdot)$ is $A$-increasing and satisfies $A$-power property, then
	\begin{equation*}\label{new0}
	\omega_{N_A}(T) \leq \left [ \frac{1}{8}N_A^2\big(T^{\sharp_A}T+TT^{\sharp_A}\big)+\frac{1}{2}\omega_{N_A}^2(T^2) \right]^{\frac{1}{4}}.
	\end{equation*}
\end{theorem}
\begin{proof}
	Let $\theta \in \mathbb{R}$. By a simple calculation, we see that
	\begin{align}\label{pintu}
	\left[\Re_A(e^{i\theta}T)\right]^4 + \left[\Im_A(e^{i\theta}T)\right]^4=\frac{1}{2}\left[\Re_A(e^{2i\theta} T^2)\right]^2+ \frac{1}{8}\left(T^{\sharp_A}T+TT^{\sharp_A}\right)^2.
	\end{align}
	Now, by applying Proposition \ref{kk1} (iii), we deduce that $\left[\Im_A(e^{i\theta}T)\right]^4\geq_A0$. Thus, \eqref{pintu} gives
	$$ \frac{1}{2}[\Re_A(e^{2i\theta} T^2)]^2+ \frac{1}{8}\left(T^{\sharp_A}T+TT^{\sharp_A}\right)^2\geq_A\left[\Re_A(e^{i\theta}T)\right]^4.$$
	Since $N_A(\cdot)$ is $A$-increasing, we get that
	\begin{align*}
	N_A\Big(\Re_A^4(e^{i\theta}T)\Big)
	&\leq N_A\Big(\frac{1}{2}\Re_A^2(e^{2i\theta} T^2)+ \frac{1}{8}\left(T^{\sharp_A}T+TT^{\sharp_A}\right)^2\Big)\\
	&\leq \frac{1}{2}N_A\Big(\Re_A^2(e^{2i\theta} T^2)\Big)+ \frac{1}{8}N_A\Big(\left(T^{\sharp_A}T+TT^{\sharp_A}\right)^2\Big).
	\end{align*}
	By applying the $A$-power property of $N_A(\cdot)$, we have that
	\begin{align*}
	N_A^4\Big(\Re_A(e^{i\theta}T)\Big)
	&\leq \frac{1}{2}N_A^2\Big(\Re_A(e^{2i\theta} T^2)\Big)+ \frac{1}{8}N_A^2\left(T^{\sharp_A}T+TT^{\sharp_A}\right)\\
	&\leq \frac{1}{8}N_A^2\big(T^{\sharp_A}T+TT^{\sharp_A}\big)+\frac{1}{2}\omega_{N_A}^2(T^2).
	\end{align*}
	Taking supremum over all $\theta \in \mathbb{R}$, we get
		\begin{equation*}
	\omega_{N_A}^4(T) \leq  \frac{1}{8}N_A^2\big(T^{\sharp_A}T+TT^{\sharp_A}\big)+\frac{1}{2}\omega_{N_A}^2(T^2),
	\end{equation*}
	as desired.
\end{proof}

\section{Inequalities for product and sum of operators}\label{s3}

\noindent Our starting point in the present section is the following two lemmas, which are needed in proving our next result.
\begin{lemma}\label{naself}
Let $T\in \mathbb{B}_A(\mathcal{H})$ be an $A$-selfadjoint operator. If  $N_A(\cdot)$ is A-selfadjoint invariant, then
\[\omega_{N_A}(T)= N_A(T).\]		
\end{lemma}
\begin{proof}
Since $N_A(\cdot)$ is $A$-selfadjoint invariant, then so is $\omega_{N_A}(\cdot)$. Hence,
	\begin{align*}
\omega_{N_A}(T)
&= \omega_{N_A}(T^{\sharp_A})\\
&=\displaystyle{\sup_{\theta \in \mathbb{R}}}\;
N_A\left(\frac{e^{i\theta}T^{\sharp_A}+e^{-i\theta}(T^{\sharp_A})^{\sharp_A}}{2}\right)\\
&=\displaystyle{\sup_{\theta \in \mathbb{R}}}\;
N_A\left(\frac{e^{i\theta}T^{\sharp_A}+e^{-i\theta}T^{\sharp_A}}{2}\right),
	\end{align*}
where the last equality follows from Proposition \ref{kk1} (i) since $T$ is $A$-selfadjoint. So, we get
	\begin{align*}
\omega_{N_A}(T)
&=N_A(T^{\sharp_A})\;\displaystyle{\sup_{\theta \in \mathbb{R}}}\;
\left|\frac{e^{i\theta}+e^{-i\theta}}{2}\right|\\
&=N_A(T^{\sharp_A})=N_A(T).
	\end{align*}
\end{proof}

\begin{lemma}\label{naself2}
Let $T\in \mathbb{B}_A(\mathcal{H})$. If $N_A(\cdot)$ is $A$-selfadjoint invariant, then
\[\omega_{N_A}(T)=\omega_{N_A}(P_{\overline{\mathcal{R}(A)}}T)=\omega_{N_A}(TP_{\overline{\mathcal{R}(A)}}).\]		
\end{lemma}
\begin{proof}
Since $N_A(\cdot)$ is $A$-selfadjoint invariant, then so is $\omega_{N_A}(\cdot)$. Thus, we get
	\begin{align*}
\omega_{N_A}(T)
&=\omega_{N_A}(T^{\sharp_A})\\
&=\displaystyle{\sup_{\theta \in \mathbb{R}}}\;
N_A\left(\frac{e^{i\theta}T^{\sharp_A}+e^{-i\theta}(T^{\sharp_A})^{\sharp_A}}{2}\right)\\
&=\displaystyle{\sup_{\theta \in \mathbb{R}}}\;
N_A\left(\frac{e^{i\theta}T^{\sharp_A}P_{\overline{\mathcal{R}(A)}}+e^{-i\theta}P_{\overline{\mathcal{R}(A)}}(T^{\sharp_A})^{\sharp_A}}{2}\right),\quad(\text{by }\eqref{ddi})\\
&=\displaystyle{\sup_{\theta \in \mathbb{R}}}\;
N_A\left(\frac{e^{i\theta}(P_{\overline{\mathcal{R}(A)}}T)^{\sharp_A}+e^{-i\theta}\big((P_{\overline{\mathcal{R}(A)}}T)^{\sharp_A}\big)^{\sharp_A}}{2}\right)\\
&=\omega_{N_A}\big((P_{\overline{\mathcal{R}(A)}}T)^{\sharp_A}\big)\\
&=\omega_{N_A}\big(P_{\overline{\mathcal{R}(A)}}T\big).
\end{align*}
On the other hand, we see that
$$\omega_{N_A}(TP_{\overline{\mathcal{R}(A)}})=\omega_{N_A}(P_{\overline{\mathcal{R}(A)}}T^{\sharp_A})=\omega_{N_A}(T^{\sharp_A})=\omega_{N_A}(T).$$
This completes the proof.
\end{proof}
\noindent Now, we are in a position to prove the following result.
\begin{theorem}\label{T.3.2}
Let $T, S\in\mathbb{B}_A(\mathcal{H})$. If $N_A(\cdot)$ is  $A$-selfadjoint invariant and submultiplicative, then
\begin{equation}\label{I.3.T.3.2}
\omega_{N_A}(TS) \leq N_A(T)\omega_{N_A}(S) + \frac{1}{2} \omega_{N_A}(TS \pm ST^{\sharp_A}),
\end{equation}
and
\begin{equation}\label{I.3.T.3.2b}
\omega_{N_A}(TS) \leq
N_A(S)\omega_{N_A}(T) + \frac{1}{2} \omega_{N_A}(TS \pm S^{\sharp_A}T).
\end{equation}
\end{theorem}
\begin{proof}
Let $\theta \in \mathbb{R}$. Since $\Re_A\big(e^{i\theta}(TS)\big)$ is $A$-selfadjoint, then by applying Lemma \ref{naself}, we get
\begin{align*}
N_A\Big(\Re_A\big(e^{i\theta}(TS)\big)\Big)
&= \omega_{N_A}\Big(\Re_A\big(e^{i\theta}(TS)\big)\Big)\\
& = \omega_{N_A}\left(\frac{e^{i\theta}TS + e^{-i\theta}S^{\sharp_A}T^{\sharp_A}}{2}\right)\\
& = \omega_{N_A}\left(T\frac{e^{i\theta}S + e^{-i\theta}S^{\sharp_A}}{2} + e^{-i\theta}\frac{S^{\sharp_A}T^{\sharp_A} - TS^{\sharp_A}}{2}\right)\\
& \leq \omega_{N_A}\left(T\frac{e^{i\theta}S + e^{-i\theta}S^{\sharp_A}}{2}\right) + \omega_{N_A}\left(e^{-i\theta}\frac{S^{\sharp_A}T^{\sharp_A} - TS^{\sharp_A}}{2}\right)\\
& \leq N_A\left(T\frac{e^{i\theta}S + e^{-i\theta}S^{\sharp_A}}{2}\right) + \omega_{N_A}\left(e^{-i\theta}\frac{S^{\sharp_A}T^{\sharp_A} - TS^{\sharp_A}}{2}\right),
\end{align*}
where the last inequality follows from the second inequality in \eqref{666}. Moreover, since $N_A(\cdot)$ is submultiplicative, then we get
\begin{align*}
N_A\Big(\Re_A\big(e^{i\theta}(TS)\big)\Big)
& \leq N_A(T)N_A\big(\Re_A(e^{i\theta}S)\big) + \frac{1}{2}\omega_{N_A}(S^{\sharp_A}T^{\sharp_A} - TS^{\sharp_A})\\
& \leq N_A(T)\omega_{N_A}(S) + \frac{1}{2}\omega_{N_A}(S^{\sharp_A}T^{\sharp_A} - TS^{\sharp_A})\\
& = N_A(T)\omega_{N_A}(S) + \frac{1}{2}\omega_{N_A}(S^{\sharp_A}T^{\sharp_A} - TP_{\overline{\mathcal{R}(A)}}S^{\sharp_A})\\
& = N_A(T)\omega_{N_A}(S) + \frac{1}{2}\omega_{N_A}(S^{\sharp_A}T^{\sharp_A} - P_{\overline{\mathcal{R}(A)}}TP_{\overline{\mathcal{R}(A)}}S^{\sharp_A}),
\end{align*}
where the last equality follows by applying Lemma \ref{naself2}. Thus, we obtain
\begin{align*}
N_A\Big(\Re_A\big(e^{i\theta}(TS)\big)\Big)
& \leq N_A(T)\omega_{N_A}(S) + \frac{1}{2}\omega_{N_A}(TS - ST^{\sharp_A}).
\end{align*}
Therefore, by the supremum over all $\theta \in \mathbb{R}$ in the last inequality, we obtain
\begin{align}\label{I.1.T.3.2}
\omega_{N_A}(TS)\leq N_A(T)\omega_{N_A}(S) + \frac{1}{2}\omega_{N_A}(TS - ST^{\sharp_A}).
\end{align}
Further, by replacing $T$ by $-iT$ in \eqref{I.1.T.3.2}, we get
\begin{align}\label{I.2.T.3.2}
\omega_{N_A}(TS)
& \leq N_A(T)\omega_{N_A}(S) + \frac{1}{2}\omega_{N_A}(TS + ST^{\sharp_A}).
\end{align}
By combining \eqref{I.1.T.3.2} together with \eqref{I.2.T.3.2}, we prove \eqref{I.3.T.3.2} as required.

On the other hand, by replacing $T$ by $S^{\sharp_A}$ and $S$ by $T^{\sharp_A}$ in (\ref{I.3.T.3.2}), we see that
\begin{align*}
\omega_{N_A}(TS)
&= \omega_{N_A}\big((TS)^{\sharp_A}\big) = \omega_{N_A}(S^{\sharp_A}T^{\sharp_A})\\
& \leq N_A(S^{\sharp_A})\omega_{N_A}(T^{\sharp_A}) + \frac{1}{2} \omega_{N_A}\Big(S^{\sharp_A}T^{\sharp_A} \pm T^{\sharp_A}(S^{\sharp_A})^{\sharp_A}\Big)\\
& = N_A(S)\omega_{N_A}(T) + \frac{1}{2} \omega_{N_A}(TS \pm S^{\sharp_A}T).
\end{align*}
This gives \eqref{I.3.T.3.2b} as desired and so the proof is complete.
\end{proof}

\noindent The next result provides a natural generalization of another well known theorem proved by Fong and Holbrook (see \cite[Th. 3]{fong}).
\begin{theorem}\label{hook}
Let $T, S\in\mathbb{B}_{A}(\mathcal{H})$. If $N_A(\cdot)$ is $A$-selfadjoint invariant and submultiplicative, then
\begin{align*}
\omega_{N_A}(TS\pm ST^{\sharp_A}) \leq2N_A\big(T\big)\,\omega_{N_A}(S).
\end{align*}
\end{theorem}
\begin{proof}
We first prove that
\begin{equation}\label{initial1}
\omega_{N_A}(TS+ ST^{\sharp_A}) \leq2N_A\big(T\big)\,\omega_{N_A}(S).
\end{equation}
Since $N_A(\cdot)$ is $A$-selfadjoint invariant, so by applying \eqref{prop1p}, we get
\begin{align}\label{777aaaaa}
\omega_{N_A}(TS+ ST^{\sharp_A})
& = \omega_{N_A}([TS+ ST^{\sharp_A}]^{\sharp_A})\nonumber\\
 &=\displaystyle{\sup_{\theta \in \mathbb{R}}}\,N_A\Big(\Re_A\left(e^{i\theta}[TS+ ST^{\sharp_A}]^{\sharp_A}\right)\Big).
\end{align}
Furthermore, a short calculation shows that
$$\Re_A\left(e^{i\theta}[TS+ ST^{\sharp_A}]^{\sharp_A}\right)= (T^{\sharp_A})^{\sharp_A} \left[\Re_A\left(e^{i\theta}S^{\sharp_A}\right)\right]+ \left[\Re_A\left(e^{i\theta}S^{\sharp_A}\right)\right]T^{\sharp_A},$$
for every $\theta \in \mathbb{R}$. So, by using \eqref{777aaaaa} together with Proposition \ref{0p1} we get
\begin{align*}
 &\omega_{N_A}(TS+ ST^{\sharp_A})\\
 &=\displaystyle{\sup_{\theta \in \mathbb{R}}}\,N_A\Big((T^{\sharp_A})^{\sharp_A} \left[\Re_A\left(e^{i\theta}S^{\sharp_A}\right)\right]+ \left[\Re_A\left(e^{i\theta}S^{\sharp_A}\right)\right]T^{\sharp_A}\Big)\\
  &\leq N_A\big((T^{\sharp_A})^{\sharp_A} \big)\left(\displaystyle{\sup_{\theta \in \mathbb{R}}}\,N_A\Big(\Re_A\left(e^{i\theta}S^{\sharp_A}\right)\Big)\right)+
  N_A\big(T^{\sharp_A} \big)\left(\displaystyle{\sup_{\theta \in \mathbb{R}}}\,N_A\Big(\Re_A\left(e^{i\theta}S^{\sharp_A}\right)\Big)\right)
  \\
    &=2N_A\big(T\big)\,\omega_{N_A}(S^{\sharp_A})\\
&=2N_A\big(T\big)\,\omega_{N_A}(S).
\end{align*}
This proves \eqref{initial1}. By replacing $T$ by $iT$ in \eqref{initial1} we get
\begin{align}
\omega_{N_A}(TS-ST^{\sharp_A}) \leq2N_A\big(T\big)\,\omega_{N_A}(S).
\end{align}
Thus we get the desired inequality.
\end{proof}

\noindent The following corollary is an immediate consequence of Theorems \ref{T.3.2} and \ref{hook}.
\begin{corollary}\label{cor2p}
	Let $S,T\in \mathbb{B}_A(\mathcal{H})$. If $N_A(\cdot)$ is  $A$-selfadjoint invariant and submultiplicative, then
	\begin{eqnarray*}
		\omega_{N_A}(TS)
		&\leq & 2 \min \Big\{  \omega_{N_A}(T) N_A(S), \omega_{N_A}(S) N_A(T) \Big \} \\
		&\leq&  4 \omega_{N_A}(T)\omega_{N_A}(S).
	\end{eqnarray*}
\end{corollary}
\begin{proof}
The first inequality follows by applying Theorem \ref{T.3.2} together with Theorem \ref{hook}. On the other hand, the second inequality follows from the following facts: $N_A(S)\leq 2 \omega_{N_A}(S)$ and $N_A(T)\leq 2 \omega_{N_A}(T)$.	
\end{proof}

Our next inequality reads as follows.

\begin{theorem}
	Let $T,S,X\in \mathbb{B}_A(\mathcal{H})$. If $N_A(\cdot)$ is  $A$-selfadjoint invariant and submultiplicative, then
	\begin{equation*}
	\omega_{N_A}\big(TXS\pm S^{\sharp_A}XT^{\sharp_A}\big) \leq 2N_A(T)N_A(S)\omega_{N_A}(X).
	\end{equation*}
\end{theorem}

\begin{proof}
It follows from Proposition \ref{0p1} that
	\begin{align*}\label{first550}
	&\omega_{N_A}\big(TXS+S^{\sharp_A}XT^{\sharp_A}\big)\nonumber\\
	&= \omega_{N_A}\Big(S^{\sharp_A}X^{\sharp_A}T^{\sharp_A}+(T^{\sharp_A})^{\sharp_A}X^{\sharp_A}(S^{\sharp_A})^{\sharp_A}\Big)\nonumber\\
	&=\sup_{\theta\in \mathbb{R}}N_A\Big(\Re_A\Big(e^{i\theta}\big(S^{\sharp_A}X^{\sharp_A}T^{\sharp_A}+(T^{\sharp_A})^{\sharp_A}X^{\sharp_A}(S^{\sharp_A})^{\sharp_A}\big) \Big)\Big).
	\end{align*}
	Moreover, for all $\theta\in \mathbb{R}$, we have
	\begin{align*}
	&\Re_A\Big(e^{i\theta}\big(S^{\sharp_A}X^{\sharp_A}T^{\sharp_A}+(T^{\sharp_A})^{\sharp_A}X^{\sharp_A}(S^{\sharp_A})^{\sharp_A}\big) \Big)\\
	&=S^{\sharp_A}\Re_A\big(e^{i\theta}X^{\sharp_A}\big)T^{\sharp_A}+(T^{\sharp_A})^{\sharp_A}\Re_A\big(e^{i\theta}X^{\sharp_A}\big)(S^{\sharp_A})^{\sharp_A}.
	\end{align*}
Now since $N_A(\cdot)$ is $A$-submultiplicative, so we have
	\begin{align*}
	&N_A\Big[\Re_A\Big(e^{i\theta}\big(S^{\sharp_A}X^{\sharp_A}T^{\sharp_A}+(T^{\sharp_A})^{\sharp_A}X^{\sharp_A}(S^{\sharp_A})^{\sharp_A}\big) \Big)\Big]\\
	&\leq N_A\big(S^{\sharp_A}\big)N_A\Big(\Re_A\big(e^{i\theta}X^{\sharp_A}\big)\Big) N_A\big(T^{\sharp_A}\big)\\
	 &\,\,\,\,\,\,\,\,\,\,\,\,\,\,\,\,\,\,\,\,\,\,\,+N_A\big((T^{\sharp_A})^{\sharp_A}\big)N_A\Big(\Re_A\big(e^{i\theta}X^{\sharp_A}\big)\Big)N_A\big((S^{\sharp_A})^{\sharp_A}\big)\\
	&\leq N_A\big(S^{\sharp_A}\big)\omega_{N_A}(X^{\sharp_A})N_A\big(T^{\sharp_A}\big)+N_A\big((T^{\sharp_A})^{\sharp_A}\big)\omega_{N_A}(X^{\sharp_A})N_A\big((S^{\sharp_A})^{\sharp_A}\big)\\
	&=2N_A(T)N_A(S)\omega_{N_A}(X),
	\end{align*}
	where the last equality follows from the fact that $N_A(\cdot)$ is $A$-selfadjoint invariant. By taking  supremum over all $\theta \in \mathbb{R}$, we get
	\begin{equation}\label{kitlaa}
	\omega_{N_A}\big(TXS+S^{\sharp_A}XT^{\sharp_A}\big) \leq 2N_A(T)N_A(S)\omega_{N_A}(X).
	\end{equation}
	Further, by replacing $S$ by $-iS$ in \eqref{kitlaa}, we get
	\begin{equation}\label{kitlaa1}
	\omega_{N_A}\big(TXS-S^{\sharp_A}XT^{\sharp_A}\big) \leq 2N_A(T)N_A(S)\omega_{N_A}(X).
	\end{equation}
	Combining \eqref{kitlaa} together with \eqref{kitlaa1}, we get 
	\begin{equation*}
	\omega_{N_A}\big(TXS\pm S^{\sharp_A}XT^{\sharp_A}\big) \leq 2N_A(T)N_A(S)\omega_{N_A}(X).
	\end{equation*}
	as required.
\end{proof}

Based on the above theorem we get the following corollary.

\begin{corollary}
	Let $T,S,X\in \mathbb{B}_A(\mathcal{H})$. If $N_A(\cdot)$ is  $A$-selfadjoint invariant and submultiplicative, then
	\begin{equation*}\label{s1}
	\omega_{N_A}\big(TXT^{\sharp_A}\big) \leq N_A^2(T)\omega_{N_A}(X),
	\end{equation*}
	and
	\begin{equation*}\label{s2}
	\omega_{N_A}\big(T^{\sharp_A}XT\big) \leq N_A^2(T)\omega_{N_A}(X),
	\end{equation*}
\end{corollary}
\begin{proof}
	First inequality follows by replacing $X$, $T$ and $S$ by $X^{\sharp_A}$, $(T^{\sharp_A})^{\sharp_A}$ and $T^{\sharp_A}$ in \eqref{kitlaa}, respectively. Moreover, second inequality follows from the first inequality by replacing $T$ and $X$ by $T^{\sharp_A}$ and $X^{\sharp_A}$, respectively.
\end{proof}

\section{Concrete examples}\label{s4}
\noindent In this section, we study the seminorm $\omega_{N_A}(\cdot)$ when $N_A(\cdot)$ is a particular seminorm.
 First we consider $N_A(\cdot)$ is $\|.\|_{A,{\alpha}}$, (see \cite{BSP}). Recall that, for $T\in \mathbb{B}_{A^{1/2}}(\mathcal{H})$
$$\|T\|_{A,\alpha}=\sup_{\|x\|_A=1}\sqrt{\alpha|\langle Tx,x\rangle_A|^2+(1-\alpha)\|Tx\|_A^2}.$$
 For every $T\in \mathbb{B}_A(\mathcal{H})$, we see that
\begin{eqnarray*}
\|\Re_A(T)\|_A&=&\omega_A(\Re_A(T))\,\,\Big(\textit{since $\Re_A(T)$ is $A$-selfadjoint} \Big)\\
&=& \sup_{\|x\|_A=1}|\langle \Re_A(T)x,x\rangle_A|\\
&=& \sup_{\|x\|_A=1}\sqrt{\alpha |\langle \Re_A(T)x,x\rangle_A|^2+(1-\alpha)|\langle \Re_A(T)x,x\rangle_A|^2  }\\
&\leq & \sup_{\|x\|_A=1}\sqrt{\alpha |\langle \Re_A(T)x,x\rangle_A|^2+(1-\alpha)\|\Re_A(T)x\|_A^2  }\\
&=& \|\Re_A(T)\|_{A,\alpha}.
\end{eqnarray*}
On the other hand, we have that
\begin{eqnarray*}
\|\Re_A(T)\|_{A,\alpha}&=& \sup_{\|x\|_A=1}\sqrt{\alpha |\langle \Re_A(T)x,x\rangle_A|^2+(1-\alpha)\|\Re_A(T)x\|_A^2  }\\
&\leq & \sup_{\|x\|_A=1} \|\Re_A(T)x\|_A\\
&=& \|\Re_A(T)\|_A.
\end{eqnarray*}
Hence, we have that
\begin{eqnarray}\label{p1}
\|\Re_A(T)\|_{A,\alpha} &=& \|\Re_A(T)\|_A.
\end{eqnarray}

\noindent On the basis of the identity \eqref{p1}, we prove the following proposition. Before that, we note that $\omega_{N_A}(\cdot)$ will be denoted by $\omega_{\|.\|_{A,\alpha}}(\cdot)$ when $N_A(\cdot)=\|.\|_{A,\alpha}.$
\begin{proposition}
If $T\in \mathbb{B}_A(\mathcal{H})$, then \, $\omega_{\|.\|_{A,\alpha}}(T)=\omega_A(T).$
\end{proposition}
\begin{proof}
	It follows from (\ref{p1}) that for all $\theta\in \mathbb{R}$,
	\begin{eqnarray*}
	\|\Re_A(e^{\rm i \theta}T)\|_{A,\alpha} &=& \|\Re_A(e^{\rm i \theta}T)\|_A.
	\end{eqnarray*}
Therefore, taking supremum over all $\theta\in \mathbb{R}$, we get that $\omega_{\|.\|_{A,\alpha}}(T)=\omega_A(T).$
\end{proof}

\noindent 
Next we define the following new seminorm on $\mathbb{B}_A(\mathcal{H})$.
\begin{definition}
Let $T\in \mathbb{B}_A(\mathcal{H})$. The function $\Omega_A(\cdot): \mathbb{B}_A(\mathcal{H})\to \mathbb{R}^+$ is defined as:
\begin{eqnarray*}
\Omega_A(T)=\sup\left \{ \| \alpha T+\beta T^{\sharp_A}\|_A \,;\;\; \alpha,\beta\in \mathbb{C},\; |\alpha|^2+|\beta|^2\leq 1 \right\}.
\end{eqnarray*}

\end{definition}

\noindent In the following proposition, we sum up some basic properties of the function $\Omega_A(\cdot)$. We need the following  lemma to prove the proposition.
\begin{lemma}\cite{BFP}
	Let $a,b,c \in \mathcal{H}.$ Then  
	\begin{align}\label{I.3.T.5.2.001}
		|\langle a, c\rangle_A|^2 + |\langle b, c\rangle_A|^2\leq \|c\|_A^2\Big(\max\big\{\|a\|_A^2, \|b\|_A^2\big\} + |\langle a, b\rangle_A|\Big),
	\end{align}
	for any $a, b, c \in \mathcal{H}$.
\end{lemma}
\begin{proposition}\label{oma}
\begin{itemize}
  \item [(i)] $\Omega_A(\cdot)$ defines a seminorm on $\mathbb{B}_A(\mathcal{H})$.
  \item [(ii)] For all $T\in\mathbb{B}_{A}(\mathcal{H})$, it holds
\begin{equation}\label{om00}
\Omega_A(T)=\sup\left \{ \sqrt{|\langle Tx, y\rangle_A|^2+|\langle T^{\sharp_A}x, y\rangle_A|^2}\,\,;\;x,y\in \mathcal{H},\,\|x\|_{A}=\|y\|_{A}= 1 \right\}.
\end{equation}
  \item [(iii)] $\Omega_A(\cdot)$ is $A$-selfadjoint invariant.
  \item [(iv)] For all $T\in\mathbb{B}_{A}(\mathcal{H})$, it holds
 \begin{equation}\label{om0}
\|T\|_A\leq \Omega_A(T)\leq\gamma_A(T)\leq \sqrt{2}\|T\|_A,
\end{equation}
where $\gamma_A(T)=\min\left\{\sqrt{\|TT^{\sharp_A} + T^{\sharp_A}T\|_A}, \sqrt{\|T\|_A^2 + \omega_A(T^2)}\right\}$.
  \item [(v)] If $T$ is $A$-selfadjoint, then $\Omega_A(T)=\sqrt{2}\|T\|_A$.
\end{itemize}
\end{proposition}
\begin{proof}
\noindent (i)\;Follows immediately.\\
 \noindent (ii)\; Note that for any $z_1,z_2\in \mathbb{C}$, we have
\begin{align*}
\sup\left\{\Big|\alpha z_1 + \beta z_2\Big|^2;\;(\alpha,\beta)\in \mathbb{C}^2,\;|\alpha|^2 + |\beta|^2 \leq 1 \right\} = |z_1|^2 + |z_2|^2.
\end{align*}
Now, let $x,y \in \mathcal{H}$ be such that $\|x\|_{A}=\|y\|_{A}= 1$. This implies that
\begin{align*}
&|\langle Tx, y\rangle_A|^2+|\langle T^{\sharp_A}x, y\rangle_A|^2\\
&=\sup\left\{\Big|\alpha \langle Tx, y\rangle_A + \beta \langle T^{\sharp_A}x, y\rangle_A\Big|^2;\;(\alpha,\beta)\in \mathbb{C}^2,\;|\alpha|^2 + |\beta|^2 \leq 1 \right\}\\
&=\sup_{|\alpha|^2 + |\beta|^2 \leq 1}\Big| \langle  \big(\alpha T+\beta T^{\sharp_A} \big)x, y\rangle_A\Big|^2.\\
\end{align*}
Thus we get, 
$$\sqrt{|\langle Tx, y\rangle_A|^2+|\langle T^{\sharp_A}x, y\rangle_A|^2}=\sup_{|\alpha|^2 + |\beta|^2 \leq 1}\Big| \langle  \big(\alpha T+\beta T^{\sharp_A} \big)x, y\rangle_A\Big|.$$
So, the desired result follows by taking the supremum over all $x,y \in \mathcal{H}$ with $\|x\|_{A}=\|y\|_{A}= 1$ in the last equality and then using \eqref{newsemi}.
\par \vskip 0.1 cm \noindent (iii)\;Follows trivially.
\par \vskip 0.1 cm \noindent (iv)\;Since $|\langle Tx, y\rangle_A|^2\leq |\langle Tx, y\rangle_A|^2+|\langle T^{\sharp_A}x, y\rangle_A|^2$ for all $x,y\in \mathcal{H}$, then the first inequality in \eqref{om0} follows immediately by taking the supremum over all $x,y \in \mathcal{H}$ with $\|x\|_{A}=\|y\|_{A}= 1$ and then using \eqref{newsemi} together with \eqref{om00}. Now, let $z \in \mathcal{H}$ be such that $\|z\|_A=1$. Let also $\alpha, \beta \in \mathbb{C}$ be such that $|\alpha|^2 + |\beta|^2 \leq 1$. An application of  the Cauchy--Schwarz inequality gives
\begin{align*}
\big\|\alpha Tz + \beta T^{\sharp_A}z\big\|_A
& \leq \sqrt{|\alpha|^2 + |\beta|^2}\sqrt{\|Tz\|_A^2 + \|T^{\sharp_A}z\|_A^2}\\
& \leq \sqrt{\big\langle (TT^{\sharp_A} + T^{\sharp_A}T)z, z\big\rangle}\leq \sqrt{\omega_A(TT^{\sharp_A} + T^{\sharp_A}T)}.
\end{align*}
Since $TT^{\sharp_A} + T^{\sharp_A}T\geq_A0$, then by using Proposition \ref{kk1} (ii), we get
\begin{align}\label{I.2.T.5.2.07}
\big\|\alpha Tz + \beta T^{\sharp_A}z\big\|_A \leq \sqrt{\|TT^{\sharp_A} + T^{\sharp_A}T\|_A}.
\end{align}
So, by taking the supremum over all $z\in\mathcal{H}$ with $\|z\|_A=1$ and then over all $\alpha, \beta \in \mathbb{C}$ with $|\alpha|^2 + |\beta|^2 \leq 1$ in \eqref{I.2.T.5.2.07}, we obtain
\begin{align}\label{I.2.T.5.2.0}
\Omega_A(T) \leq \sqrt{\|TT^{\sharp_A} + T^{\sharp_A}T\|_A}.
\end{align}

Now, let $x,y \in \mathcal{H}$ with $\|x\|_{A}=\|y\|_{A}= 1$. By setting $a = Tx$, $b = T^{\sharp_A}x$, $c = y$ in \eqref{I.3.T.5.2.001}, we obtain
\begin{align*}
|\langle Tx, y\rangle_A|^2 + |\langle T^{\sharp_A}x, y\rangle_A|^2\leq \|y\|_A^2\Big(\max\big\{\|Tx\|_A^2, \|T^{\sharp_A}x\|_A^2\big\} + |\langle Tx, T^{\sharp_A}x\rangle_A|\Big).
\end{align*}
This yields that
\begin{align}\label{I.3.T.5.2}
\sqrt{|\langle Tx, y\rangle_A|^2 + |\langle T^{\sharp_A}x, y\rangle_A|^2}\leq \sqrt{\max\big\{\|Tx\|_A^2, \|T^{\sharp_A}x\|_A^2\big\} + |\langle T^2x, x\rangle_A|}.
\end{align}
So, by taking the supremum over all $x,y\in \mathbb{R}$ with  $\|x\|=\|y\|= 1$ in \eqref{I.3.T.5.2} and then using \eqref{om00} we get
\begin{align}\label{I.4.T.5.2}
\Omega_A(T) \leq  \sqrt{\|T\|_A^2 + \omega_A(T^2)}.
\end{align}
By combining \eqref{I.4.T.5.2} together with \eqref{I.2.T.5.2.0}, we prove the second inequality in \eqref{om0}.
\par \vskip 0.1 cm \noindent (v)\;Since $T$ is $A$-selfadjoint, then so is $T^{\sharp_A}$ and $(T^{\sharp_A})^{\sharp_A}=T^{\sharp_A}$. Thus, we see that
\begin{align*}
\Omega_A(T)
&=\Omega_A(T^{\sharp_A})\\
&=\sup\left \{ \sqrt{|\langle T^{\sharp_A}x, y\rangle_A|^2+|\langle T^{\sharp_A}x, y\rangle_A|^2}\,\,;\;x,y\in \mathcal{H},\,\|x\|_{A}=\|y\|_{A}= 1 \right\}\\
&=\sqrt{2}\|T^{\sharp_A}\|_A=\sqrt{2}\|T\|_A.
\end{align*}
Hence, we complete the proof.
\end{proof}

\noindent For the rest of this section, we will take the seminorm $\Omega_A(\cdot)$ instead of $N_A(\cdot)$. In this case, the seminorm $\omega_{N_A}(\cdot)$ will be denoted by $\omega_{\Omega_A}(\cdot)$. Hence,
$$\omega_{\Omega_A}(T) = \displaystyle{\sup_{\theta\in \mathbb{R}}}\,\Omega_A\big(\,\Re_A(e^{i\theta}T)\big),$$
for any $T\in\mathbb{B}_A(\mathcal{H})$. Since $\Omega_A(\cdot)$ is  $A$-selfadjoint invariant, so by \eqref{666}, we deduce that
\begin{align}\label{b1}
\frac{1}{2}\Omega_A(T)\leq \omega_{\Omega_A}(T)\leq \Omega_A(T),\quad\forall\,T\in \mathbb{B}_A(\mathcal{H}).
\end{align}
The following lemma gives a formula for $\omega_{\Omega_A}(T)$ in terms of $\omega_A(T)$ when $T\in \mathbb{B}_A(\mathcal{H})$.
\begin{lemma}\label{T.7.2.0}
	Let $T\in\mathbb{B}_A(\mathcal{H})$. Then
$$\omega_{\Omega_A}(T) = \sqrt{2}\omega_A(T).$$
\end{lemma}
\begin{proof}
Since $\Re_A(e^{i\theta}T)$ is $A$-selfadjoint for each $\theta\in \mathbb{R}$, then it follows from the item $(5)$ of Proposition \ref{oma} that $\Omega_A\left(\Re_A(e^{i\theta}T)\right)=\sqrt{2}\|\Re_A(e^{i\theta}T)\|_A$. So, we see that
\begin{align*}
\omega_{\Omega_A}(T)
&= \displaystyle{\sup_{\theta\in \mathbb{R}}}\,\Omega_A\big(\,\Re_A(e^{i\theta}T)\big)\\
&= \sqrt{2}\,\displaystyle{\sup_{\theta\in \mathbb{R}}}\,\big\|\,\Re_A(e^{i\theta}T)\big\|
= \sqrt{2}\omega_A(T).
\end{align*}
\end{proof}

\noindent In the following theorem, we give a sufficient condition under which the equality $\omega_{\Omega_A}(T)=\Omega_A(T)$ holds.
\begin{theorem}\label{an}
Let $T\in\mathbb{B}_A(\mathcal{H})$ be an $A$-normal operator. Then
$$\omega_{\Omega_A}(T) = \Omega_A(T).$$
\end{theorem}
\begin{proof}
Since $T$ is $A$-normal, so by \cite{feki01}, we get
\begin{equation}\label{a1}
\omega_A(T) = \|T\|_A \quad\text{and}\quad \omega_A(T^2)=\omega_A^2(T).
\end{equation}
Now, by using Lemma \ref{T.7.2.0} together with the second inequality in \eqref{b1}, one may prove that $\omega_A(T) \leq \frac{\sqrt{2}}{2}\Omega_A(T)$. This implies, through \eqref{om0}, that
\begin{align}\label{a2}
\omega_A(T) \leq \frac{\sqrt{2}}{2}\Omega_A(T) \leq \frac{\sqrt{2}}{2}\min\left\{\sqrt{\|TT^{\sharp_A} + T^{\sharp_A}T\|_A}, \sqrt{\|T\|_A^2 + \omega_A(T^2)}\right\}.
\end{align}
So, since $T$ is $A$-normal, then an application of \eqref{a1} and \eqref{a2} shows that $\Omega_A(T) = \sqrt{2}\|T\|_A$. Therefore, by Lemma \ref{T.7.2.0} together with \eqref{a1}, we infer that
$$\omega_{\Omega_A}(T) = \sqrt{2}\omega_A(T) = \sqrt{2}\|T\|_A = \Omega_A(T).$$
This completes the proof of the theorem.
\end{proof}
\begin{remark}
We note that the converse of Theorem \ref{newsemi} is, in general, not true. Indeed, let $A = \begin{pmatrix}1&0&0\\0&1&0\\0&0&1\end{pmatrix}$ and $T=\begin{pmatrix}0&1&0\\0&0&0\\0&0&2\end{pmatrix}$. One may check that $\omega_{\Omega}(T)=\Omega(T)=2 \sqrt{2}$. But $T$ is not normal.
\end{remark}

In the following theorem, we present an equivalent condition
for $\omega_{\Omega_A}(\cdot)=\frac{1}{2}\Omega_A(\cdot)$.
\begin{theorem}\label{T.3.5.2}
Let $T\in\mathbb{B}_A(\mathcal{H})$. Then, the following assertions are equivalent:
\begin{itemize}
\item[(i)] $\omega_{\Omega_A}(T) = \frac{1}{2}\Omega_A(T)$.
\item[(ii)] $\Omega_A(T) = 2\sqrt{2}\|\,\Re_A(e^{i\theta}T)\|_A$ for all $\theta \in \mathbb{R}$.
\end{itemize}
\end{theorem}
\begin{proof}
$(ii)\Rightarrow (i)$ Assume that $\Omega_A(T) = 2\sqrt{2}\|\,\Re_A(e^{i\theta}T)\|_A$ for all $\theta \in \mathbb{R}$.
So, by applying Lemma \ref{T.7.2.0}, we get
\begin{align*}
\frac{1}{2}\Omega_A(T) = \sqrt{2}\,\displaystyle{\sup_{\theta\in \mathbb{R}}}\,\big\|\,\Re_A(e^{i\theta}T)\big\|_A
= \sqrt{2}\omega_A(T) = \omega_{\Omega_A}(T).
\end{align*}
This shows that $\omega_{\Omega_A}(T) = \frac{1}{2}\Omega_A(T)$.

$(i)\Rightarrow (ii)$ Suppose that $\omega_{\Omega_A}(T) = \frac{1}{2}\Omega_A(T)$. Let $\theta\in \mathbb{R}$. We see that
\begin{align*}
\Omega_A(T)
 &= \Omega_A(e^{i\theta}T)= \Omega_A\big(\Re_A(e^{i\theta}T) + i\,\Im_A(e^{i\theta}T)\big)\\
 & \leq \Omega_A\big(\Re_A(e^{i\theta}T)\big) + \Omega_A\big(\Im_A(e^{i\theta}T)\big)\\
&\leq \omega_{\Omega_A}(T) + \omega_{\Omega_A}(T) = \Omega_A(T).
\end{align*}
This implies that
\begin{align}\label{T.I.1.3.5.2}
\Omega_A\big(\Re_A(e^{i\theta}T)\big) + \Omega_A\big(\Im_A(e^{i\theta}T)\big) = \Omega_A(T).
\end{align}
Moreover, one observes that
\begin{align*}
\frac{1}{2}\Omega_A(T)
&= \omega_{\Omega_A}(T)\\
&\geq \max\Big\{\Omega_A\big(\Re_A(e^{i\theta}T)\big), \Omega_A\big(\Im_A(e^{i\theta}T)\big)\Big\}\\
& = \frac{\Omega_A\big(\Re_A(e^{i\theta}T)\big) + \Omega_A\big(\Im_A(e^{i\theta}T)\big)}{2}+ \frac{\Big|\Omega_A\big(\Re_A(e^{i\theta}T)\big) - \Omega_A\big(\Im_A(e^{i\theta}T)\big)\Big|}{2}\\
& \geq \frac{\Omega_A\Big(\Re_A(e^{i\theta}T) + i\,\Im_A(e^{i\theta}T)\Big)}{2}
+ \frac{\Big|\Omega_A\big(\Re_A(e^{i\theta}T)\big) - \Omega_A\big(\Im_A(e^{i\theta}T)\big)\Big|}{2}\\
& = \frac{\Omega_A(e^{i\theta}T)}{2} + \frac{\Big|\Omega_A\big(\Re_A(e^{i\theta}T)\big) - \Omega_A\big(\Im_A(e^{i\theta}T)\big)\Big|}{2}
\geq \frac{1}{2}\Omega_A(T).
\end{align*}
This shows that
\begin{align}\label{T.I.2.3.5.2}
\Omega_A\big(\Re_A(e^{i\theta}T)\big) = \Omega_A\big(\Im_A(e^{i\theta}T)\big).
\end{align}
Therefore, a combination of \eqref{T.I.1.3.5.2} and \eqref{T.I.2.3.5.2} gives $\Omega_A(T) = 2\Omega_A\big(\Re_A(e^{i\theta}T)\big)$. Finally, since $\Re_A(e^{i\theta}T)$ is $A$-selfadjoint, then  Proposition \ref{oma} (v) gives us
$\Omega_A(T) = 2\sqrt{2}\|\,\Re_A(e^{i\theta}T)\|_A$ as required.
\end{proof}

\end{document}